\documentclass[11pt]{amsart}

\usepackage{hyperref}
\usepackage{amssymb,amsfonts}
\usepackage{hyperref}
\usepackage{amssymb,textcomp}
\usepackage{enumerate}

\usepackage[utf8]{inputenc}
\usepackage[T1]{fontenc}

\usepackage[mathscr]{eucal}

\usepackage[a4paper, twoside=false, vmargin={2cm,3cm}, includehead]{geometry}

\newtheorem{lemma}{Lemma}[section]
\newtheorem{theorem}[lemma]{Theorem}
\newtheorem*{theorem*}{Theorem}
\newtheorem{corollary}[lemma]{Corollary}

\newtheorem{proposition}[lemma]{Proposition}
\newtheorem*{proposition*}{Proposition}

\newtheorem*{problem*}{Problem}

\theoremstyle{definition}
\newtheorem*{claim*}{Claim}

\newtheorem*{definition}{Definition}

\newtheorem*{remark}{Remark}
\newtheorem*{remarks}{Remarks}

\newcommand{\E}{{\mathbb E}}

\newcommand{\N}{{\mathbb N}}

\newcommand{\Q}{{\mathbb Q}}
\newcommand{\R}{{\mathbb R}}

\newcommand{\T}{{\mathbb T}}
\newcommand{\Z}{{\mathbb Z}}

\newcommand{\CC}{{\mathcal C}}

\newcommand{\CF}{{\mathcal F}}
\newcommand{\CI}{{\mathcal I}}
\newcommand{\CK}{{\mathcal K}}

\newcommand{\CX}{{\mathcal X}}
\newcommand{\CY}{{\mathcal Y}}

\newcommand{\CZ}{{\mathcal Z}}

\newcommand{\bN}{{\mathbf{N}}}

\newcommand{\ve}{\varepsilon}

\newcommand{\wt}{\widetilde}
\newcommand{\e}{\mathrm{e}}% for the function exp
\newcommand{\one}{\mathbf{1}}

\newcommand{\norm}[1]{\left\Vert #1\right\Vert}
\newcommand{\nnorm}[1]{\lvert\!|\!| #1|\!|\!\rvert}

\newcommand{\inv}{^{-1}}

\DeclareMathOperator{\id}{id}

\newcommand{\Krat}{{\CK_{\text{\rm rat}}}}

\begin{document}
	
	\title[Multiple recurrence and convergence without commutativity]{Multiple recurrence and convergence without commutativity} %%along $n,n^2$}	

	\thanks{The first author was supported  by the Research Grant - ELIDEK HFRI-FM17-1684.}
\author{Nikos Frantzikinakis}
\address[Nikos Frantzikinakis]{University of Crete, Department of mathematics and applied mathematics, Voutes University Campus, Heraklion 71003, Greece} \email{frantzikinakis@gmail.com}
\author{Bernard Host}
\address[Bernard Host]{
	LAMA, Univ Gustave Eiffel, CNRS, F-77447 Marne-la-Vall\'ee, France}
\email{bernard.host@u-pem.fr}
	\begin{abstract}
 	We establish multiple recurrence and convergence results for pairs of zero entropy measure preserving transformations that do not satisfy any commutativity assumptions. Our results  cover   the case where the iterates of the  two transformations  are $n$ and $n^k$ respectively, where $k\geq 2$, and the case $k=1$ remains an open problem.  Our starting point
 	is based on the observation  that  Furstenberg systems of sequences of the form $(f(T^{n^k}x))$ have very special structural properties when  $k\geq 2$. We  use these properties and some  disjointness arguments in order to get characteristic factors  with nilpotent structure for the corresponding ergodic averages, and then finish the proof using some  equidistribution results on nilmanifolds.
	\end{abstract}


\maketitle

\section{Introduction and main results}
 A well known multiple recurrence  result of Furstenberg and Katznelson \cite{FuK79} states that if $T,S$ are commuting,  measure preserving transformations,  acting on a probability space
 $(X,\CX,\mu)$, then  for every set $A\in \CX$ with $\mu(A)>0$, there exists $n\in \N$ such that
 \begin{equation}\label{E:TSrec}
 \mu(A\cap T^{-n}A\cap S^{-n}A)>0.
 \end{equation}
 This result was extended by Bergelson and Leibman to cover iterates given by  arbitrary integer polynomials with zero constant terms \cite{BL96},  and the case where
 $T$ and $S$ generate a nilpotent group \cite{BL02}.   For $f,g\in L^\infty(\mu)$, the existence in $L^2(\mu)$ of the limit
 $$
\lim_{N\to\infty}\frac 1N\sum_{n=1}^NT^nf \cdot S^ng
$$
was established in the commutative case by Conze and Lesigne~\cite{CL}, and the extension to any number of transformations spanning a nilpotent group and to polynomial iterates was established by Walsh~\cite{W}.
  The goal of this article is to  study similar recurrence  phenomena and  mean convergence results for pairs of measure preserving transformations that do not satisfy any commutativity assumptions.

 We  caution the reader at this point that a simple example of Furstenberg \cite[Page~40]{Fu} shows that if $T$ and $S$ do not satisfy any commutativity assumptions, then the recurrence property  \eqref{E:TSrec} may  fail in general.
 Indeed, if $T$ is the $(1/2,1/2)$-Bernoulli shift on the sequence space $X=\{0,1\}^\Z$, and $R$ is the transformation on $X$ that fixes the  $0^\text{th}$-coordinate and  flips the $0$'s and the $1$'s in all  other  coordinates, then for
 $$
 A:=\{x\in \{0,1\}^\Z\colon x(0)=1\} \ \text{ and } \  S:=R^{-1}TR,
 $$ an easy computation shows that $\mu(T^{-n}A\cap S^{-n}A)=0$ for every $n\in\N$.
  Other examples, which also cover non-convergence results and different iterates,  are given in~\cite[Example~7.1]{Be85},  \cite{BL04}, and \cite[Lemma~4.1]{FLW11}. Furthermore,   Proposition~\ref{th:entropy} below shows that this situation is fairly general.
%
%
%
%By  adapting  this example and another of by  Berend~\cite[Example~7.1]{Be85}, one can show (see  \cite[Lemma~4.1]{FLW11}), that  if $a, b\colon \N\to \N$ are arbitrary strictly increasing sequences,
%$T$,
%then  the sequence
%$$
%\mu(T^{-a(n)}A\cap S^{-b(n)}A)
%$$
%can take arbitrary values in $\{0,1/4\}$ (see also~\ref{th:entropy} and~\cite{BL04} for related results).
So in the absence of additional assumptions, there is no hope for  recurrence and convergence results
that apply to pairs of general, not necessarily commuting, measure preserving  transformations $T,S$.

\subsection{Results}
On the positive side, it was recently established  in \cite{F21}  that if the transformation $T$ has zero entropy (in the above mentioned  counterexamples $T$ and $S$ have positive entropy), then for any positive non-integer $a$ we have for all sets $A\in \CX$ with $\mu(A)>0$ that there exists $n\in\N$ such that
 \begin{equation}\label{E:arec}
  \mu(A\cap T^{-n}A\cap S^{-[n^{a}]}A)>0.
 \end{equation}
 It was also established in  \cite{F21} that  the corresponding multiple ergodic averages
 \begin{equation}\label{E:aconv}
 \frac{1}{N}\sum_{n=1}^N T^nf \cdot  S^{[n^a]}g
\end{equation}
converge in $L^2(\mu)$ as $N\to\infty$ and a simple formula for the limit was obtained.
 The method of \cite{F21} makes essential use of the fact that the power of $n$ in the second iterate  is non-integral;  this  hypothesis gives  access to properties of strongly stationary systems that the so called  ``Furstenberg systems'' of sequences of the form $(g(S^{[n^a]}x))$ satisfy. Such information is lost when $a$ is integral, and it was left as an open problem (see \cite[Problem~2]{F21}) whether the recurrence property \eqref{E:arec}, or the mean convergence of the averages \eqref{E:aconv},
 still holds when  $a\in \N$. The main objective of this article is  to give a positive answer to these questions  when $a\in \N$ is different than $1$. The case where  $a=1$ remains open (see Section~\ref{SS:problems} for additional related problems).

\begin{theorem}
\label{th:convergence}
Let $T,S$ be measure preserving transformations acting on a probability space $(X,\CX,\mu)$  such that the system $(X,\mu,T)$ has zero entropy.  Let also   $p\in \Z[t]$ be a polynomial with $\deg(p)\geq 2$. Then
for every    $f,g\in L^\infty(\mu)$ the limit
\begin{equation}
\lim_{N\to\infty} \frac{1}{N}\sum_{n=1}^N T^nf \cdot  S^{p(n)}g
\label{eq:lim-main}
\end{equation}
exists in  $L^2(\mu)$.
\end{theorem}
The rational Kronecker factor   $\Krat(T)$ is the factor spanned by eigenfunctions associated to rational eigenvalues of $T$. In the case where $T=S$ and $p(n)=n^2$, Furstenberg and Weiss~\cite{FW} proved that this factor is characteristic, in the sense of~\cite{FW},  for mean convergence of
the averages \eqref{eq:lim-main}, a result that was later extended to the case of commuting transformations in
\cite{CFH}. In the context of Theorem~\ref{th:convergence},   the methods used in these articles do not allow us to get similar characteristic factors for the averages \eqref{eq:lim-main}. We use  a different approach to establish the following result.
\begin{theorem}
\label{prop:Rationa-Kro}
Under the hypothesis of Theorem~\ref{th:convergence}, the limit~\eqref{eq:lim-main} in $L^2(\mu)$ is equal to $0$
if either $\E(f|\Krat(T))=0$
or $\E(g|\Krat(S))=0$.
\end{theorem}
 From this result it is an easy mater to deduce the following.
\begin{theorem}
\label{th:recurrence}
Let $T,S$ be measure preserving transformations acting on a probability space $(X,\CX,\mu)$  such that the system $(X,\mu,T)$ has zero entropy.
Let also   $p\in \Z[t]$ be a polynomial with $\deg(p)\geq 2$ and $p(0)=0$.
Then for every $A\in\CX$ and  $\ve>0$, the set
\begin{equation}\label{E:returns}
\{n\in\N\colon\mu(A\cap T^{-n}A\cap S^{-p(n)}A)\geq \mu(A)^3-\varepsilon\}
\end{equation}
has positive lower density.
\end{theorem}
\begin{remarks}
$\bullet$ Even proving that
$\mu(A\cap T^{-n}A\cap S^{-p(n)}A)>0$ for some $n\in \N$  is not much easier than the stronger statement mentioned above.

$\bullet$ It is possible to prove that the set \eqref{E:returns} is syndetic but we do not address this problem here as this would complicate our statements and proofs.
\end{remarks}
When $T=S$ or $T,S$ commute, this result  was obtained without an entropy assumption in \cite{FrK06} and \cite{CFH} respectively.

%%We would also like to remark that
The convergence and recurrence results  of Theorems~\ref{th:convergence}-\ref{th:recurrence} are  non-trivial even in the case of  distal or weakly mixing systems, the reason being that in the absence of any commutativity assumptions, the  traditional methods for proving recurrence and mean convergence results do not seem to give us a useful starting point.

\subsection{Methods}
Our approach has some similarities with the one used  in \cite{F21} to  study the averages \eqref{E:aconv} for non-integral values of $a$. We proceed  by  first studying  Furstenberg systems of sequences of the form $(f(T^{p(n)}x))$; the advantage being that this  problem  can be handled by   analyzing    multiple ergodic averages that involve  only a single transformation, a problem that has been well studied  and  we can employ an arsenal of tools. A crucial difference with the approach taken  in \cite{F21}, is that unlike the case of  iterates given by non-integral powers, systems arising from polynomial iterates are not in general strongly stationary. Nevertheless, when the polynomial $p$ is non-linear  the  corresponding Furstenberg systems are expected to  have very particular structure, and  in fact, it was conjectured in \cite{F21} that they are direct products of Bernoulli systems and systems of algebraic structure. Although, we do not  establish this structural result, we obtain some partial information in Proposition~\ref{P:PolynomialIterates} that suffices for our purposes. It enables us
to  use a disjointness argument  and reduce the study of the limiting behavior of the averages in \eqref{eq:lim-main} to the case where the transformation $S$ is   structured (see  Proposition~\ref{prop:char3}). We then proceed by using  a variant of the
main structural result in \cite{HK1} and a Nilsequence Wiener-Wintner convergence result from \cite{HK3}, to conclude  the proof of Theorem~\ref{th:convergence} (see Proposition~\ref{prop:main}). Theorem~\ref{prop:Rationa-Kro}, as well as Theorem~\ref{th:recurrence}, require some additional work. After further reducing matters to the case where both transformations  $T$ and $S$ have algebraic structure,
we are left with verifying  an equidistribution property for nilsystems (stated in Proposition~\ref{P:mainnil}). After this property is established, it is an easy matter to show that the rational Kronecker factor is a characteristic factor for mean convergence and from this we deduce the recurrence property of Theorem~\ref{th:recurrence}.

\subsection{Is zero entropy necessary?}
It is natural to ask if the zero entropy assumption on $T$ in the previous results is necessary.
The answer is yes and in fact, by adapting known examples, we will  show that if $T$ has positive entropy, then for some measure preserving transformation $S$   the conclusions of Theorems~\ref{th:convergence}-\ref{th:recurrence} fail. This is a direct consequence of the following result (which is a variant of  \cite[Lemma~4.1]{FLW11}).
\begin{proposition}
\label{th:entropy}
Let $(X,\mu,T)$ be an ergodic system with positive entropy, $a,b\colon \N\to \Z\setminus \{0\}$ be injective sequences that miss infinitely many integers, and $F$ be an arbitrary subset of $\N$.
Then there exist a system $(X,\mu,S)$, with positive entropy, a
measurable set $A$, and $c>0$, such that
$$
\mu(T^{-a(n)}A\cap S^{-b(n)}A)=\begin{cases}0 \quad \text{if } n\in F,\\
c  \quad \text{if } n\notin F.\end{cases}
$$
As a consequence, there exist a  transformation $S$ and a set $A$ with $\mu(A)>0$   so that
$\mu(T^{-a(n)}A\cap S^{-b(n)}A)=0$ for every $n\in\N$, and another transformation $S$ and  a set $A$  so that
the  averages $\frac{1}{N}\sum_{n=1}^N\mu(T^{-a(n)}A\cap S^{-b(n)}A)$
do not converge as $N\to\infty$.
\end{proposition}
\begin{remark}
Theorems~\ref{th:convergence} and \ref{th:recurrence} show that if the   transformation $T$ has zero entropy, then	 such general constructions are not possible.
	\end{remark}

\subsection{Further directions}
\label{SS:problems}
Our arguments fail if in place of the iterates $n,p(n)$ in Theorems~\ref{th:convergence}-\ref{th:recurrence}
 we use the pair of iterates $n,n$ or the pair $n^2,n^3$. In the first case the problem
 is that  Furstenberg systems of sequences of the form $(g(S^nx))$ no
 longer have special structure, and we  therefore lose the crucial starting point provided by  Proposition~\ref{P:PolynomialIterates}. In the second case, the
 problem is that the zero entropy assumption on $T$ does not pass to   Furstenberg systems of the sequence $(f(T^{n^2}x))$ (which can be Bernoulli).
%,  in fact, even if $(X,\mu,T)$ is a $2$-step distal system,.The Furstenberg system  of  the last sequence  may be Bernoulli  for almost every $x\in X$
In either case,
 we were not able to prove convergence or recurrence
  even when both $T$ and $S$ are weakly mixing or isomorphic to  $2$-step distal systems of the form $(x,y)\mapsto (x+\alpha, y+f(x))$ defined on $\T^2$ with the   Lebesgue measure $m_{\T^2}$. This leads to the following problem.
  \begin{problem*}
 Does Theorem~\ref{th:convergence} hold if in place of the iterates $n,p(n)$
 we use the pair of iterates $n,n$ or the pair $n^2,n^3$? Do we also have recurrence in these cases?
  \end{problem*}
One can ask a more general question for pairs of iterates given by arbitrary polynomials $p,q\in \Z[t]$ with $p(0)=q(0)=0$.
   It is not clear which way the answer should go and we would not be surprised if the answer to these questions turns out to be  negative (see though \cite{A85} for some positive results when the iterates are $n,n$).

Lastly, we expect that under the assumptions of Theorem~\ref{th:convergence}  the averages \eqref{eq:lim-main}
converge pointwise almost everywhere. In this regard, we have the following conditional result that lends credence to this  conjecture.
\begin{proposition}
\label{th:pointwise}
Let $T,S$ be measure preserving transformations acting on a probability space $(X,\CX,\mu)$  such that the system $(X,\mu,T)$ has zero entropy.
	Let  also  $p\in \Z[t]$ with $\deg(p)\geq 2$, and suppose that for  all $\ell\in \N$ and  functions $f_1,\ldots, f_\ell\in L^\infty(\mu)$, the averages
	\begin{equation}\label{E:assumption}
	\frac{1}{N}\sum_{n=1}^N f_1(S^{p(n+1)} x)\cdot \ldots \cdot f_\ell(S^{p(n+\ell)}x)
	\end{equation}
	converge pointwise almost everywhere. Then for  all functions
	$f,g\in L^\infty(\mu)$, the averages
	$$
	\frac{1}{N}\sum_{n=1}^N f(T^{n} x)\cdot  g(S^{p(n)}x)
	$$
	converge pointwise almost everywhere and they converge to $0$ if either $\E(f|\Krat(T))=0$ or $\E(g|\Krat(S))=0$.
\end{proposition}
We remark that pointwise convergence of  the averages \eqref{E:assumption} is expected to hold, but it is considered  to be very hard to establish  even when $\ell=2,3$ and $p(n)=n^2$.
%%\textbf{[Write: For simplicity in the sequel we consider only real valued functions]}

\subsection{Notation} \label{SS:notation}
For  $N\in\N$  we let  $[N]:=\{1,\dots,N\}$.
We usually denote sequences on $\N$ or on $\Z$ by  $(a(n))$, instead of $(a(n))_{n\in\N}$ or $(a(n))_{n\in \Z}$; the domain of the sequence is going to be clear from the context.
  With $\T$ we denote the one dimensional torus  $\R/\Z$, and we often identify it   with $[0,1)$.
We often denote elements of $\T$ with  real numbers and we are implicitly  assuming that these real numbers are taken  modulo $1$. For $t\in \R$ or $\T$ we let $\e(t):=\exp(2\pi i t)$.
\subsection*{Acknowledgement} The authors would like to thank the referee for helpful remarks and corrections.

\section{Background and  tools from ergodic theory}
In this section we gather some basic facts needed in the sequel and establish some notation. In
order to avoid unnecessary repetition,  we refer the reader to \cite{HK-book} and~\cite{G} for some other standard notions from ergodic theory used in this article.

\subsection{Measure preserving systems}
Throughout,  by a {\em measure preserving system}, or simply a {\em system},  we mean a Lebesgue probability space $(X,\mathcal{X},\mu)$ together with an invertible, measurable,  and measure preserving transformation $T\colon X\to X$.  In general we omit  the $\sigma$-algebra and write a system as $(X,\mu,T)$ and sometimes we abbreviate the notation and write $X$ or $T$.
The system is ergodic if the only $T$-invariant sets in $\CX$ have measure $0$ or $1$. If $f\in L^\infty(\mu)$ and $n\in\Z$, with $T^nf$ we denote the composition $f\circ T^n$, where for $n\in\N$, we let $T^n:=T\circ\cdots\circ T$ ($n$ times),  $T^{-n}=(T\inv)^n$,  and $T^0=\id$.

\subsection{Factors and joinings}
\label{subsec:joinings}
Let $(X,\CX,\mu,T)$ be a system. By a \emph{factor} of $X$ we mean  a $T$-invariant sub-$\sigma$-algebra of $\CX$.
A \emph{factor map} from the  system  $(X,\CX,\mu,T)$ to the system $(Y,\CY,\nu, S)$ is a measurable map $\pi\colon X\to Y$ with $\pi\circ T=S\circ\pi$ and such that $\nu$ is the image of the measure $\mu$ under $\pi$. In this case, the sub-$\sigma$-algebra $\pi\inv(\CY)$ of $\CX$
is a factor of $X$ and every factor of $X$ can be obtained in this way.
We say that a factor {\em is spanned} by a given collection of functions if it coincides  with the smallest $\sigma$-algebra with respect to which  all these functions are measurable.

A \emph{joining} of two systems $(X,\mu,T)$ and $(Y,\nu,S)$ is a measure $\rho$ on  $X\times Y$, invariant under $T\times S$ and whose projections on $X$ and $Y$ are equal to $\mu$ and $\nu$, respectively.

\subsection{The rational Kronecker factor}
\label{SS:RK}
Let $(X,\mu,T)$ be a system.
For $r\in\N$ we let $\CI(T^r)$ be the $\sigma$-algebra of $T^r$-invariant subsets of $X$. The \emph{rational Kronecker factor} of the system is the $\sigma$-algebra
$\bigvee_{r\in \N}\CI(T^r)$.
We denote it by  $\Krat(T)$ or  $\Krat(X,\mu,T)$.
It follows from the definition that this factor  is  spanned by the  functions $f\in L^2(\mu)$ with $Tf=\e(s)f$ for some $s\in\Q$.
It follows also that for $f\in L^2(\mu)$ we have $\E_\mu(f|\Krat(T))=0$ if and only if $\E_\mu(f|\CI(T^r))=0$ for every $r\in\N$, that is,
$$
 \lim_{N\to\infty}\frac{1}{N}\sum_{n=1}^N f(T^{rn}x)=0 \ \text{ for every } r\in\N,\ \mu\text{-a.e..}
$$
Hence, if  $\mu=\int\mu_x\,d\mu(x)$ is the ergodic decomposition of $\mu$, then  $\E_\mu(f|\Krat(X,\mu,T))=0$ if and only if $\E_{\mu_x}(f|\Krat(X,\mu_x,T))=0$ for $\mu$-almost every $x\in X$.

\subsection{The Pinsker factor}
Every system $(X,\mu,T)$ admits a largest factor belonging to the class of systems of entropy zero. This factor is called the \emph{Pinsker factor} of $(X,\mu,T)$ and is denoted by  $\Pi(X,\mu,T)$, or $\Pi(T)$,  if there is no danger of  confusion.  The next result follows from  \cite[Lemma 3]{T75}
(see also \cite{dlR} and ~\cite[Proposition~2.2]{KKLR}).
%% The class of systems of entropy zero is \emph{characteristic} in %%the sense of~\cite{KKLR}, meaning that it is invariant under %%countable joinings and taking factors. We copy here %%Proposition~2.2 of this paper for this particular case.
\begin{proposition}
\label{prop:Pinsker}
Let $\rho$ be a  joining of  a system $(X,\mu,T)$ of zero entropy  and a system $(Y,\nu,S)$. If $g\in L^2(\nu)$ is such that $\E_\nu(g| \Pi(S))=0$, then for every $f\in L^2(\mu)$ we have
$\int f(x)\cdot g(y)\,d\rho(x,y)=0$.
\end{proposition}

We deduce from  this  the following fact (which is  probably a folklore result  but we were not able to find a textbook reference).
\begin{corollary}
\label{cor:Pinsker}
Let $\pi\colon(X,\mu,T)\to(Y,\nu,S)$ be a factor map and $g\in L^2(\nu)$. If $g$ is measurable with respect to $\Pi(S)$, then $g\circ\pi$ is measurable with respect to $\Pi(T)$. If $\E_\nu(g|\Pi(S))=0$,  then $\E_\mu(g\circ\pi|\Pi(T))=0$.
\end{corollary}
\begin{proof}
The first assertion follows immediately from the maximality of $\Pi(T)$ in the family of factors of $X$ of entropy $0$. We prove the second assertion.
 If $g\in L^2(\nu)$ satisfies  $\E_\nu(g|\Pi(S))=0$ and $f\in L^2(\mu)$ is measurable with respect to $\Pi(T)$, it suffices to prove that $\int f\cdot g\circ\pi\,d\mu=0$. We denote  the Pinsker factor of $(X,\mu,T)$ by  $(\Pi(T),\lambda,R)$, with factor map $p\colon X\to\Pi(T)$. Then the function $f$ can be written as $f=h\circ p$ for some function $h\in L^2(\lambda)$. Let $\rho$ be the measure on $\Pi(T)\times Y$ that is the  image of the measure $\mu$ under the map $p\times\pi$. Then $\rho$ is a joining of  $(\Pi(T),\lambda,R)$ and $(Y,\nu,S)$ and by  Proposition~\ref{prop:Pinsker} we have $\int h(z)\cdot g(y)\,d\rho(z,y)=0$. By the definition of $\rho$ this integral is equal to $\int f\cdot g\circ\pi\,d\mu$ and we are done.
\end{proof}
The following simple fact will be crucial in the proof of Proposition~\ref{P:PolynomialIterates} below.
\begin{lemma}
\label{lem:Pinsker}
Let $(X,\mu,T)$ be a system and $f\in L^\infty(\mu)$ be real valued. Suppose that for every $\ell\geq 0$ and all $n_1,\dots,n_\ell\in\N$, distinct or not, we have
$$
\int f\cdot\prod_{j=1}^\ell T^{n_j}f\,d\mu=0.
$$
Then $\E_\mu(f|\Pi(T))=0$.
\end{lemma}
(Here and below  the empty product is equal to $1$ by convention and thus the integral above is equal to $\int f\,d\mu$ if $\ell=0$.)
\begin{remark}
	If the function $f$ is complex valued, then in order to get the same conclusion we have to
	consider products of iterates of $f$ and $\bar{f}$ on the assumption.
\end{remark}
\begin{proof}
	Let $\mathcal{Y}$ be the  $\sigma$-algebra spanned  by $\{T^nf \colon n\in \Z\}$.
	%%, meaning, the  $\sigma$-algebra spanned by all $T^nf$-measurable sets for $n\in \Z$.
	Let also $Y$ be the corresponding factor-system. By Corollary~\ref{cor:Pinsker} it suffices to prove that $f$ has zero conditional expectation on the Pinsker factor of $Y$. Therefore, after substituting $Y$ for $X$, we can assume that $\mathcal{Y}=\mathcal{X}$.

Let $\CF$ be the $\sigma$-algebra spanned by $\{T^kf\colon  k\geq 0\}$.
For $n\geq 0$, $T^n\CF$ is the $\sigma$-algebra spanned by $\{T^kf\colon k\geq -n\}$ and by our reduction we have
$\CF_\infty:=\bigvee_{n=0}^\infty T^n\CF=\CX$.
On the other hand, for $n\in \N$ we have that $T^{-n}\CF$ is the $\sigma$-algebra
spanned by $\{T^kf\colon k\geq n\}$. By \cite[Lemma~18.7]{G}, the $\sigma$-algebra
$\CF_{-\infty}:=\bigcap_{n=0}^\infty T^{-n}\CF$ contains the Pinsker $\sigma$-algebra $\Pi(T)$ of $X$.
By hypothesis, $f$ has zero conditional expectation on $T^{-1}\CF$ and thus
on the smaller $\sigma$-algebra  $\CF_{-\infty}$. Since $\CF_{-\infty}$ contains $\Pi(T)$  we are done.
 \end{proof}

 \subsection{Nilsystems and nilsequences}
 A \emph{nilmanifold}  is a homogeneous space $X=G/\Gamma$ where  $G$ is a nilpotent Lie group,
 and $\Gamma$ is a discrete cocompact subgroup of $G$. If $G_{k+1}=\{e\}$, where $G_k$ denotes the $k$-th commutator subgroup of $G$, we say that $X$ is a
 $k$-\emph{step nilmanifold}.

Let $X=G/\Gamma$ be a  $k$-step nilmanifold. The   group $G$ acts on $X$ by left
 translations and we write this action as $(g\cdot x)\mapsto g\cdot x$ for $g\in G$ and $x\in X$.
  The unique probability measure on $X$ invariant under this action is called the {\em  Haar measure of $X$} and is denoted by  $m_X$.
If  $a\in G$,  $x\in X$, and
 $f\in \CC(X)$,
following~\cite{BHK} we call the sequence $(f(a^n\cdot x))$ a \emph{basic $k$-step nilsequence}.
 A \emph{$k$-step nilsequence}, is a uniform limit of \emph{basic $k$-step nilsequences}.
 If $a\in G$ and $T\colon X\to X$ is the map $x\mapsto a\cdot x$,  then
 the system $(X=G/\Gamma, m_X, T)$ is called a {\it $k$-step  nilsystem}. For  classical properties of nilsystems see~\cite[Chapter 11]{HK-book}.

 \subsection{Systems of finite order}\label{SS:ok}
 Let $(X,\mu,T)$ be a system. For $k\in\N$, the definition of the seminorm $\nnorm\cdot_k$ on $L^\infty(\mu)$, as well as related facts recorded  in this subsection, can be found in Chapters~8 and~9 of~\cite{HK-book}. We write $\nnorm\cdot_{\mu,k}$ in case of ambiguity. We define also $\nnorm f_\infty:=\sup_{k\in\N}\nnorm f_k$ and the properties of this seminorm follow immediately from the properties of the seminorms of finite index.
 For $k\in\N\cup\{\infty\}$  we say that  $(X,\mu,T)$ is a \emph{system  of order $k$} if the seminorm $\nnorm\cdot_{k+1}$ is a norm. For $k\in\N\cup\{\infty\}$, each system $(X,\mu,T)$  admits  a maximal factor of order $k$, that we denote by $\CZ_k(X,\mu,T)$ or $\CZ_k(T)$   when there is danger of confusion.
  Each system of order $\infty$ is an inverse limit of systems of finite orders, and for every system  we have $\CZ_\infty(T)=\bigvee_{k\in\N}\CZ_k(T)$.

 Let  $(X,\mu,T)$ be a system,  $\mu=\int\mu_x\,d\mu(x)$ be the ergodic decomposition of $\mu$ under $T$, and  $k\in\N\cup\{\infty\}$. Then $(X,\mu,T)$ is a system of order $k$, if and only if $(X,\mu_x,T)$ is a system of order $k$ for almost every $x\in X$. For $k\in\N$ this is~\cite[Chapter 9, Proposition~23]{HK-book} and the case $k=\infty$ follows.

 For $k\in\N$, the structure theorem~\cite[Chapter~16, Theorem~1]{HK-book}, originally proved in~\cite{HK1}, states  that ergodic systems of order $k$  are exactly the inverse limits of ergodic $k$-step  nilsystems.

\begin{theorem}[Nilsequence Wiener-Wintner~\cite{HK3}]
\label{th:WW}
Let $(X,\mu,T)$ be a system and $f\in L^\infty(\mu)$. Then there exists a subset $X_0$ of $X$ with $\mu(X_0)=1$ such that the limit
\begin{equation}
\label{eq:WW}
\lim_{N\to\infty}\frac 1N\sum_{n=1}^N \Psi(n)\cdot f(T^nx)
\end{equation}
exists for every $x\in X_0$ and every nilsequence $(\Psi(n))$.
Moreover, if $\E_\mu(f|\CZ_\infty(T))=0$, then the limit~\eqref{eq:WW} is equal to $0$ for every $x\in X_0$ .
\end{theorem}
\begin{remarks}
$\bullet$ It is  important  for our purposes  that the set $X_0$ does not depend on the nilsequence $\Psi$.

$\bullet$ The first assertion is~\cite[Theorem 2.22]{HK3} in the case of an ergodic system, and
\cite[Chapter~23, Theorem~14]{HK-book} in the general case. The second assertion is implicit in~\cite{HK3} and~\cite{HK-book} and we give a short proof here.
\end{remarks}
\begin{proof}
Let $k\in\N$. We write $\bN$ for the sequence of intervals $([N])_{N\in \N}$. The proof
 uses the notion of the uniformity seminorms $\norm\cdot_{\bN,k+1}$ for sequences in $\ell^\infty(\Z)$ that was introduced in~\cite{HK3}.

  If $\E_\mu(f|\CZ_\infty(T))=0$ we will show that there exists a subset $X_0$ of $X$ with $\mu(X_0)=1$ such that the limit~\eqref{eq:WW} is equal to $0$ for every $x\in X_0$  and for every $k$-step nilsequence $\Psi$.

Let  $\mu=\int\mu_x\,d\mu(x)$ be the ergodic decomposition of $\mu$ and
for $x\in X$ let $\phi_x\in\ell^\infty(\Z)$ be the sequence $(f(T^nx))$.
Since $\E_\mu(f|\CZ_\infty(T))=0$ we have for every $k\in\N$ that $\nnorm f_{\mu,k+1}=0$, and since $\int \nnorm f_{\mu_x,k+1}^{2^{k+1}}\,d\mu(x)=\nnorm f_{\mu,k+1}^{2^{k+1}}=0$, it follows that $\nnorm f_{\mu_x,k+1}=0$ for $\mu$-almost every $x\in X$.
By~\cite[Corollary 3.10]{HK3}, for $\mu$-almost every $x\in X$ and  $\mu_x$-almost every $x'\in X$ we have
$\norm{\phi_{x'}}_{\bN,k+1}=\nnorm f_{\mu_x,k+1}=0$. Therefore, $\norm{\phi_x}_{\bN,k+1}=0$ for a set  $X_0$  of $x\in X$ that has  full measure.

Furthermore, by \cite[Corollary 2.14]{HK3}, for every $\delta>0$ there exists a constant $C=C(\delta,\Psi)>0$ such that for every bounded sequence $\phi$ we have
 $$
\limsup_{N\to\infty} \Bigl|\frac 1N\sum_{n=1}^N \Psi(n) \cdot \phi(n)\Bigr|\leq C\norm{\phi}_{\bN,k+1}+\delta\norm{\phi}_\infty.
$$
Applying this with $\phi=\phi_x$  for $x\in X_0$ (recall that $\norm{\phi_x}_{\bN,k+1}=0$ for  every $x\in X_0$) and letting $\delta\to 0$, gives the announced result.
\end{proof}

\subsection{Furstenberg systems of sequences}\label{SS:Defs}
We reproduce here  the notion of a Furstenberg system of a bounded sequence
from \cite{FH18}  and record some basic related facts that will be used later. For notational convenience we restrict to real valued sequences.
\begin{definition}
	Let   $ \bN:=([N_k])_{k\in\N}$ be a sequence of intervals with $N_k\to \infty$ and $I$ be a bounded closed interval in $\R$.
	We say that the sequence $z=(z(n))_{n\in\Z}$ with values in $I$
	{\em admits   correlations  on $\bN$}, if the   limits
	\begin{equation}\label{E:CorDef}
	\lim_{k\to\infty}\frac{1}{N_k}\sum_{n=1}^{N_k} \prod_{j=1}^s z(n+n_j)
	\end{equation}
	exist for all $s \in \N$ and all   $n_1,\ldots, n_s\in \Z$
	(not necessarily distinct).
\end{definition}
\begin{remarks}
$\bullet$	 Given $z\colon \Z \to I$,   using a diagonal argument, we get that every sequence of intervals $\bN=([N_k])_{k\in \N}$
	has a subsequence $\bN'=([N_k'])_{k\in\N}$, such that the sequence  $z$  admits  correlations on $\bN'$.

$\bullet$
If we are given a one-sided sequence $(z(n))_{n\in\N}$ we extend it to $\Z$ in an arbitrary way; then the existence and values of the correlations do not depend on the extension.
\end{remarks}

If a sequence $z\colon \Z\to I$   admits  correlations on a given sequence of intervals,  then we use a  variant of the correspondence principle of Furstenberg~\cite{Fu77, Fu} in order
to associate a measure preserving system that captures the statistical properties of this sequence. We briefly describe this process next.

 Let $\Omega:=I^\Z$. The elements of $\Omega$ are written $\omega:=(\omega(n))_{n\in\Z}$ and the shift $\sigma\colon\Omega\to\Omega$ is defined by $(\sigma \omega)(n):=\omega(n+1)$,  $n\in \Z$. We consider the sequence $z$ as an element of $\Omega$. Note that  the algebra generated by functions of the form $\omega\mapsto \omega(k)$, $\omega\in \Omega$, for   $k\in  \Z$,  separates points in $\Omega$. We conclude that  if the
sequence $z\colon \Z\to I$
admits   correlations  on $\bN$, then  for all $f\in \CC(\Omega)$ the following limit exist
$$
\lim_{k\to\infty}\frac{1}{N_k}\sum_{n=1}^{N_k}\,  f(\sigma^nz).
$$
Hence, the  following weak*-limit exists
$$
\nu:=\lim_{k\to\infty}\frac{1}{N_k}\sum_{n=1}^{N_k}\, \delta_{\sigma^nz}
$$
and we say that \emph{the point $z$ is generic for $\nu$ along $\bN$}.

\begin{definition}
\label{D:correspondence}
Let $I$ be a  compact interval and let $z\colon \Z\to I$ be a sequence that  admits
correlations  on $\bN:=([N_k])_{k\in\N}$, and  $\Omega$,  $\sigma$,
and $\nu$ as above.
\begin{itemize}
\item
We call $(\Omega,\nu,\sigma)$ the {\em  Furstenberg system
associated with} $z$ on  $\bN$.

\item
We let $F_0\in \CC(\Omega)$  be defined by $F_0(\omega):=\omega(0)$, $\omega \in \Omega$, and call it the {\em $0^{\text{th}}$-coordinate projection}. Note that  $F_0(\sigma^nz)=z(n)$ for every $n\in \Z$ and
\begin{equation}
\label{E:correspondence}
\lim_{k\to\infty}\frac{1}{N_k}\sum_{n=1}^{N_k}\,  \prod_{j=1}^s z(n+n_j) =\int  \prod_{j=1}^s
\sigma^{n_j}F_0 \, d\nu
\end{equation}
for all  $s\in \N$, $n_1, \ldots, n_s\in \Z$. This identity  is often referred to as the \emph{Furstenberg correspondence principle}.
\item
 We say that the sequence
$z$ has a {\em unique} Furstenberg system,  if  $z$ admits  correlations on $([N])_{N\in\N}$,  or equivalently, if $z$ is generic for a measure along $([N])_{n\in\N}$.
\end{itemize}
\end{definition}	
\begin{remark}
 A sequence $z\colon \Z\to I$ may have several non-isomorphic  Furstenberg systems
depending on which sequence of intervals $\bN$ we use in the evaluation of its   correlations. We call any such system {\em a  Furstenberg system of $z$}.
\end{remark}

The following fact is an immediate consequence of the pointwise ergodic theorem.
\begin{proposition}\label{P:ergodic}
	Let $(X,\mu,T)$ be a system and suppose that   $\mu=\int \mu_x\, d\mu(x)$ is the ergodic decomposition of $\mu$. Then for every  real valued $f\in L^\infty(\mu)$
	and for almost every $x\in X$, the sequence $(f(T^nx))$ has a unique  Furstenberg system that is ergodic and  a factor of the system $(X,\mu_x,T)$.
\end{proposition}
(The factor map is the map $x\mapsto (f(T^nx))_{n\in\Z}$ from $X$ to $\Omega$.)
\begin{remark}
	  It follows that if $(X,\mu,T)$ has zero entropy, then
	for almost every $x\in X$, the Furstenberg system of the sequence $(f(T^nx))$ has zero entropy.
	\end{remark}

\section{Furstenberg systems of polynomial iterates}\label{S:3}
Let $(X,\mu,T)$ be a system, $f\in L^\infty(\mu)$, and $p\in \Z[t]$ with $\deg(p)\geq 2$.
In this section we are going to establish a crucial property that Furstenberg systems of sequences of the form $(f(T^{p(n)}x))$ satisfy.
A key ingredient is the following result.%% of~\cite{HK2}.
\begin{theorem}[\cite{HK2}]
\label{T:HK}
	Let $p\in \Z[t]$ be a polynomial with $\deg(p)\geq 2$.
Let also $(X,\mu,T)$ be an ergodic system, $\ell\in\N$, and  $f_1,\ldots, f_\ell \in L^\infty(\mu)$. Then the limit
$$
\lim_{N\to\infty}\frac{1}{N}\sum_{n=1}^NT^{p(n+1)}f_1\cdot\ldots\cdot T^{p(n+\ell)}f_\ell
$$
exists in $L^2(\mu)$ and it is zero if  $\E(f_j|\CZ_\infty(T))=0$ for some $j\in \{1,\ldots, \ell\}$.
\end{theorem}
We will combine  this result with Lemma~\ref{lem:Pinsker} in order to deduce the following.
\begin{proposition}\label{P:PolynomialIterates}
	Let $p\in \Z[t]$ be a polynomial with $\deg(p)\geq 2$,  $(X,\mu,T)$ be a system, and $f\in L^\infty(\mu)$ be real valued.
	Then every strictly increasing sequence of positive integers $(N_k)$ has a subsequence $(N_k')$ such that  for almost every $x\in X$,  the sequence $(f(T^{p(n)}x))$ admits correlations  on $\bN':=([N'_k])_{k\in\N}$.  Moreover, if $\E(f|\CZ_\infty(T))=0$, then  the $0^{\text{th}}$-coordinate projection  $F_0\colon\Omega\to\R$ has zero conditional expectation on the Pinsker factor of the corresponding Furstenberg system $(\Omega,\nu_x,\sigma)$.
\end{proposition}
\begin{remarks}	
$\bullet$ By taking a suitable  countable dense collection  of functions  in $L^2(\mu)$  and using a diagonal and an approximation argument, we
can choose the subsequence $(N_k')$ independently of the function $f$. We cannot choose it independently of the system or the full measure set of $x\in X$ though.

$\bullet$ It seems likely that for almost every $x\in X$ the Furstenberg systems of
 $(f(T^{p(n)}x))$   along  $\bN':=([N'_k])_{k\in\N}$ are ergodic and isomorphic to  direct products of  infinite-step nilsystems and Bernoulli systems (see \cite[Problem~1]{F21}). Proving this seems  non-trivial though.
	\end{remarks}

\begin{proof}
By Theorem~\ref{T:HK}, for all $s\in\N$ and  $h_1,\dots,h_s\in\Z$ (not necessarily distinct) the limit
\begin{equation}
\label{eq:limit}
\lim_{k\to\infty}\frac 1{N_k}\sum_{n=1}^{N_k} \prod_{j=1}^s f(T^{p(n+h_j)}x)
\end{equation}
 exists in $L^2(\mu)$.
Since mean convergence of a sequence of functions implies pointwise almost everywhere convergence along a subsequence, we can use a diagonal argument to find a subsequence $(N_k')$ of $(N_k)$ such that for a subset $X_0$ of $X$ with full measure and for $x\in X_0$ all the limits~\eqref{eq:limit} exist. This means that for $x\in X_0$ the sequence $(f(T^{p(n)}x))$ admits correlations on the sequence of intervals $\bN'=([N_k'])_{k\in\N}$, which proves
the first assertion.

  For $x\in X_0$ let $(\Omega,\nu_x,\sigma)$ be the associated Furstenberg system and $F_0\colon \Omega\to\R$ be the $0^{\text{th}}$-coordinate projection.
Note also that  the limit
\begin{equation}
	\label{eq:limit2}
	\lim_{k\to\infty}\frac{1}{N'_k}\sum_{n=1}^{N'_k}f(T^{p(n)}x)\cdot \prod_{j=1}^\ell f(T^{p(n+n_j)}x)
\end{equation}
exists for all $x\in X_0$, $\ell\geq 0$,  and  $n_1,\ldots, n_\ell\in \N$ (not necessarily distinct).
 Suppose  that  $\E(f|\CZ_\infty(T))=0$. Then by  Theorem~\ref{T:HK}  we have that for every $x$ in
a full measure subset $X_1$ of $X_0$ the limit~\eqref{eq:limit2} is equal to $0$.
Using the correspondence principle~\eqref{E:correspondence} we  get that
$$
\int F_0 \cdot  \prod_{j=1}^\ell \sigma^{n_j}F_0\,  d\nu_x=0
$$
for all $\ell\geq 0$  and  $n_1,\ldots, n_\ell\in \N$.

By Lemma~\ref{lem:Pinsker} this implies that for $x\in X_1$, the function  $F_0$ is orthogonal to the Pinsker factor of the Furstenberg system $(\Omega,\nu_x,\sigma)$. This completes the proof.
\end{proof}
We remark that  if $f$ takes values in $\{-1,1\}$, then the previous argument also gives  that for almost every $x\in X$ the Furstenberg system of $(f(T^{p(n)}x))$ on $([N_k'])_{k\in\N}$ is Bernoulli. But this is no longer necessary if say
$f$ takes values in $\{0,\pm 1\}$. To see this, consider the ergodic system defined by the transformation $T\colon \T^2\to \T^2$ given by $T(x,y):=(x+\alpha,2y)$, $x,y\in \T$,  where $\alpha$ is irrational. We cut the unit square in four  square pieces in the natural way, and let $f$  be $1$ on the first, $-1$ on the fourth, and $0$ on the other two vertically adjacent squares. Then $f$ is orthogonal to the $\CZ_\infty(T)$-factor. Suppose that the Furstenberg system of  $(f(T^{n^2}x))$ was Bernoulli.  Then  the Furstenberg system of  $(f^2(T^{n^2}x))$ would be a factor of this system.  But $f^2$ is a non-trivial measurable function  with respect to the Kronecker factor of the system $(X,\mu, T)$. This easily implies that  the Furstenberg system of $(f^2(T^{n^2}x))$ is a non-trivial factor of the affine system $(x,y)\mapsto (x+\alpha,y+x)$, and any such factor cannot be a factor of a Bernoulli system (since it is not even weak mixing).

\section{Convergence}\label{S:4}
\label{sec:convergence}
The goal of this section is to prove  Theorem~\ref{th:convergence} and to determine some convenient (but not optimal) characteristic factors for the mean convergence of the averages \eqref{eq:lim-main}. We do this in Proposition~\ref{prop:main} below.
\subsection{A partial  characteristic factor} Our first step is to combine Proposition~\ref{prop:Pinsker} with Proposition~\ref{P:PolynomialIterates}  in order to  establish  the following result via a disjointness argument.
\begin{proposition}
\label{prop:char3}
		Let $(X,\CX,\mu)$ be a probability space and $T,S\colon X\to X$ be measure preserving transformations such that the system $(X,\mu,T)$ has zero entropy. Let also   $p\in \Z[t]$ be a polynomial with $\deg(p)\geq 2$ and $f,g\in L^\infty(\mu)$ with  $\E(g|\CZ_\infty(S))=0$. Then
\begin{equation}\label{E:zero}
\lim_{N\to\infty} \frac{1}{N}\sum_{n=1}^N T^nf \cdot  S^{p(n)}g =0
\end{equation}
in  $L^2(\mu)$.
\end{proposition}
(The case where $\E(f|\CZ_\infty(T))=0$ will be treated in the next subsection.)
\begin{proof} Note first that  our assumption is satisfied for the real  and imaginary part of $g$ in place of $g$. Hence,
	we can restrict  to the case where the functions $f$ and $g$ are real valued with range on a closed bounded interval $I$.

Arguing by contradiction, suppose that the conclusion fails. Then there exist  $\varepsilon>0$ and  $N_k\to \infty$ such that
\begin{equation}\label{E:varepsilon}
\norm{\frac 1{N_k}\sum_{n=1}^{N_k} T^nf\cdot  S^{p(n)}g}_{L^2(\mu)}\geq \varepsilon
\end{equation}
for every $k\in \N$.
By  Proposition~\ref{P:PolynomialIterates} there exists a subsequence $(N_k')$ of $(N_k)$ such that for $\mu$-almost every $x\in X$  the  sequence $(g(S^{p(n)}x))$ admits correlations on $\bN':=([N'_k])$.
Let $(\Omega,\nu_x,\sigma)$ be the corresponding Furstenberg system. Moreover, since $\E(g|\CZ_\infty(S))=0$,
  by Proposition~\ref{P:PolynomialIterates} again we have that for  $\mu$-almost every $x\in X$  the  $0^{\text{th}}$-coordinate projection $F_0\colon\Omega\to\R$  satisfies $\E_{\nu_x}(F_0| \Pi(\Omega,\nu_x,\sigma))=0$.

% and that, in the  $(\Omega,\mu_x,\sigma)$, the  has zero conditional expecation in the Pinsker factor $\Pi(\Omega,\mu_x,\sigma)$.

Furthermore,  by Proposition~\ref{P:ergodic}, we have that for $\mu$-almost  every $x\in X$ the sequence  $(f(T^nx))$ admits correlations on the sequence of intervals $\bN'$ and the
	corresponding Furstenberg system $(\Xi,\nu'_x,\sigma)$ is ergodic and has zero entropy. We write $G_0\colon\Xi\to\R$ for the $0^{\text{th}}$-coordinate projection.
	
Let $X_0$ be  a subset of $X$ with $\mu(X_0)=1$  and such  that the previous properties hold for all $x\in X_0$.
	We claim that %%, under the first or the second hypothesis,
\begin{equation}\label{E:zero'}
\lim_{k\to\infty} \frac 1{N'_k}\sum_{n=1}^{N'_k} f(T^nx)\cdot  g(S^{p(n)}x)=0 \quad \text{ for almost every } x\in X_0.
	\end{equation}
	If this is shown, then we get a contradiction from \eqref{E:varepsilon} using the bounded convergence theorem.
	
	So suppose that \eqref{E:zero'} fails. Then  there exists a subset $X_1$ of $X_0$ with $\mu(X_1)>0$ such that for every  $x\in X_1$ there exists a subsequence $(N_{x,k}')$ of $(N_k')$  such that
\begin{equation}\label{E:nonzero}
	\lim_{k\to\infty} \frac 1{N'_{x,k}} \sum_{n=1}^{N'_{x,k}} f(T^nx)\cdot  g(S^{p(n)}x)\quad
	\text{ exists and is non-zero.}
\end{equation}

Let $x\in X_1$ be fixed for the moment.
Let $w,z\in I^\Z$ be defined by $w(n):=f(T^nx)$ and $z(n):=g(S^{p(n)}x)$ for $n\in\Z$. The sequence $(N'_{x,k})$ admits  a subsequence $(N''_{x,k})$ along which the point  $(w,z)$ of $I^\Z\times I^\Z$ is generic on $(I^\Z\times I^\Z,\sigma\times\sigma)$ for some measure $\rho_x$ on this space. This means that for every continuous function $\Phi$ on $I^\Z\times I^\Z$ we have
$$
\lim_{k\to\infty}\frac 1{N''_{x,k}}\sum_{n=1}^{N''_{x,k}} \Phi(\sigma^nw,\sigma^nz)=\int\Phi\,d\rho_x.
$$
In particular, since $F_0$ is continuous on $\Omega$ and $G_0$ is continuous on $\Xi$, and since $G_0(\sigma^nw)=f(T^nx)$ and $F_0(\sigma^nz)=g(S^{p(n)}x)$, $n\in \N$, we obtain
\begin{equation}
\label{eq:generic}
\lim_{k\to\infty}\frac 1{N''_{x,k}}\sum_{n=1}^{N''_{x,k}} f(T^nx)\cdot g(S^{p(n)}x) =\int G_0\otimes F_0\,d\rho_x.
\end{equation}

Note that the measure $\rho_x$ is invariant under $\sigma\times\sigma$. Furthermore, since $w$ is generic for the measure $\nu'_x$ on $(\Xi,\sigma)$ along $([N])$ and since $z$ is generic for the measure $\nu_x$ on $(\Omega,\sigma)$ along $([N'_k])$,  and since $([N''_{x,k}])$ is a subsequence of $([N'_k])$, the two coordinate projections of $\rho_x$ are $\nu'_x$ and $\nu_x$ respectively, and $\rho_x$ is a joining of the systems $(\Xi,\nu'_x,\sigma)$ and $(\Omega,\nu_x,\sigma)$.

Recall that for $\mu$-almost every $x\in X_1$ we have $\E_{\nu_x}(F_0|\Pi(\Omega,\nu_x,\sigma))=0$ and the system $(\Xi,\nu'_x,\sigma)$ has entropy zero.
By Proposition~\ref{prop:Pinsker} we have
$$
\int G_0\otimes F_0\,d\rho_x=0
$$ for $\mu$-almost every $x\in X_1$. Taken together with  identity~\eqref{eq:generic}, this gives  a contradiction by~\eqref{E:nonzero} (recall that $(N''_{k,x})$ is a subsequence of $(N'_{k,x})$). This proves~\eqref{E:zero'} and finishes the proof.
%%\textbf{[The proof is by double contradiction. Should be logically simplified]}
\end{proof}

\subsection{Convergence}
\label{subsec:convergence}
We are now ready to prove Theorem~\ref{th:convergence}. For convenience we restate it and also record some additional information that will be needed later.
\begin{proposition}
\label{prop:main}
Let $T,S$ be measure preserving transformations acting on a probability space $(X,\CX,\mu)$  such that the system $(X,\mu,T)$ has zero entropy.  Let also   $p\in \Z[t]$ be a polynomial with $\deg(p)\geq 2$.
Then for every  $f,g\in L^\infty(\mu)$ the limit
$$
\lim_{N\to\infty} \frac{1}{N}\sum_{n=1}^NT^nf \cdot  S^{p(n)}g
$$
exists in  $L^2(\mu)$.  Moreover, the  limit is $0$ if either
$\E_\mu(f|\CZ_\infty(T))=0$ or $\E_\mu(g|\CZ_\infty(S))=0$.
\end{proposition}
\begin{proof}
We first prove mean convergence. By Proposition~\ref{prop:char3} we can restrict to the case where $g$ is measurable with respect to $\CZ_\infty(S)$ and using an $L^2(\mu)$ approximation argument we can assume that it is measurable with respect to the factor  $\mathcal{Z}_k(S)$ for some $k\in \N$.

By~\cite[Proposition 3.1]{CFH} (see also~\cite[Chapter~16, Theorem~10]{HK-book}), for every $\ve>0$ there exists $\wt g\in L^\infty(\mu)$ such that
\begin{enumerate}
\item
$\wt g$ is measurable with respect to $\CZ_k(S)$ and $\norm{g-\wt g}_{L^1(\mu)}<\ve$;
\item
\label{it:tilde-g}
for $\mu$-almost every $x\in X$ the sequence $(\wt g(S^nx))$ is a $k$-step nilsequence.
\end{enumerate}
Therefore, it suffices to show the existence in $L^2(\mu)$ of the limit
\begin{equation}
\label{eq:main}
\lim_{N\to\infty}\frac 1N\sum_{n=1}^N \wt g(S^{p(n)}x)\cdot f(T^nx),
\end{equation}
where $\wt g$ satisfies~\eqref{it:tilde-g}, and that this limit is $0$ if $\E_\mu(f|\CZ_\infty(T))=0$.

By~\cite[Proposition 3.14]{L}, for $\mu$-almost every $x\in X$ the sequence $(\wt g(S^{p(n)}x))$ is an $\ell$-step nilsequence for some $\ell=\ell(k,p)$.
Let $X_0$ be the  full measure subset of $X$ for which this property holds.
Let also $X_1$ be the subset of $X$ associated with $f$ by Theorem~\ref{th:WW}. Then for every $x\in X_0\cap X_1$ the limit~\eqref{eq:main}
exists. It follows that this limit exists in $L^2(\mu)$ and the first part of the proposition is proved.
Furthermore,
if $\E_\mu(f|\CZ_\infty(T))=0$, then by the second part of Theorem~\ref{th:WW} the limit is equal to $0$ for every $x\in X_0\cap X_1$. This completes the proof.
\end{proof}

\section{The rational Kronecker factor is characteristic}\label{S:5}
The goal of this section is to prove
Theorem~\ref{prop:Rationa-Kro}. Here is a brief sketch of our strategy. Using Proposition~\ref{prop:main} we can reduce matters to the case where the functions $f$ and $g$ are measurable with respect to the factors $\CZ_k(T)$ and $\CZ_k(S)$ respectively, for some $k\in \N$.  Using an ergodic decomposition argument and  the  Structure Theorem~\cite{HK1}  for systems  of order  $k$, we can further reduce matters to proving
the equidistribution property for nilsystems  stated in Proposition~\ref{P:mainnil}. In order to establish this property we verify first that it holds for a particular class of nilsystems, namely, for ergodic unipotent affine transformations acting on some finite dimensional torus. We do this in  Lemma~\ref{L:TR}  and our starting point is the simple  observation of Lemma~\ref{L:PQR}. We then combine  Lemma~\ref{L:TR}
 with the criterion  in~\cite[Theorem~2.17]{L}, which enables us to reduce equidistribution properties of general nilsystems to the unipotent affine case,  in order to finish the proof of Proposition~\ref{P:mainnil} for general nilsystems.
\subsection{Equidistribution for unipotent affine transformations}
In what follows,  we say that the real numbers $\alpha_1,\dots,\alpha_r$ are {\em rationally independent},  if the equation $n_1\alpha_1+\dots+n_r\alpha_r=0\bmod 1$ has no non-trivial integer solutions.

As usual, for $t\in\R$ we write $\e(t)=\exp(2\pi it)$ and we use the same notation for $t\in\T=\R/\Z$.
For $d\in\N$, $x\in\T^d$, and $k\in\Z^d$, we write $k\cdot x=k_1x_1+\dots+k_dx_d$. Let $u=(u(n))$ be a sequence with values in $\T^d$. Then
by the  Weyl Equidistribution Theorem, this sequence is equidistributed in $\T^d$ if and only if
$$
\lim_{N\to\infty}\frac 1N\sum_{n=1}^N \e\bigl(k\cdot u(n))=0
\quad \text{ for every non-zero }\ k\in\Z^d.
$$
If the sequence $u$ has polynomial coordinates this condition is equivalent to saying that,
for every non-zero $k\in\Z^d$, the polynomial $n\mapsto k\cdot u(n)$ has at least one irrational non-constant coefficient.
%\begin{lemma}
%\label{lem:easy}
%Let $d\in\N$ and let $(u(n))_{n\in\N}$ be a sequence in $\T^d$ with polynomial coordinates and $p\in\Z[t]$ a non constant polynomial. If the sequence $(u(n))_{n\in\N}$ is equidistributed in $\T^d$ then the sequence $(u\circ p(n))_{n\in\N}$ is equidistributed in $\T^d$.
%\end{lemma}
%\begin{proof}
%By Weyl Equidistribution Theorem it suffices to show that when $P\in\R[t]$ has at least one irrational non constant coefficient and $p\in\Z[t]$ is non constant then $P\circ p$ has at least one non constant irrational coefficient.
%
%Let $d=\deg(p)$ and let $m\in\N$.  The map sending the vector of constant coefficients of real polynomial $P$ of degree $\leq m$  to the vector of non constant coefficients of the polynomial $P\circ p$ is linear and one to one because $p$ is not constant. This map is represented by a $dm\times m$ matrix of rank $m$, with integer entries because $p$ has integer coefficients. This matrix admits a right inverse with rational entries. Therefore,
% the coefficients of $P$ are linear combinations with rational coefficients of the coefficients of $P\circ p$. In particular, if the coefficients of $P\circ p$ are rational then the coefficients of $P$ are rational.
%\end{proof}

\begin{lemma}
\label{L:PQR}
Let $P\in \R[t]$ be a non-constant polynomial with rationally independent coefficients,
$Q\in \R[t]$ be arbitrary, and $p\in \Z[t]$ be such that $\deg(p)\geq 2$.
Then the polynomial $P+Q\circ p$ has at least one irrational non-constant coefficient.
\end{lemma}
\begin{remark}
Note that the conclusion is false if  $\deg(p)=1$.
\end{remark}
\begin{proof}
If the non-constant coefficients of $Q$ are rational, then the non-constant coefficients of the polynomial $Q\circ p $ are rational and thus the polynomial $P+Q\circ p$ has rationally independent non-constant coefficients and we are done. So suppose that $Q$ has at least one irrational non-constant coefficient. We write $Q(t)=t^{d+1}Q_1(t)+Q_2(t)$ where $d\in \N$,  $Q_1\in\Q[t]$, $\deg(Q_2)=d$,  and $Q_2$ has an irrational leading coefficient.  Then the polynomial $p(t)^{d+1} (Q_1\circ p)(t)$ has rational coefficients. After replacing $P(t)$ with $P(t)+p(t)^{d+1}(Q_1\circ p)(t)$ and $Q$ with $Q_2$, we are reduced to the case where $Q$ has an irrational leading coefficient.

We distinguish two cases. If $\deg(Q)\geq\deg(P)$, then $\deg(Q\circ p)>\deg(Q)\geq\deg(P)$, because $\deg(p)\geq 2$ and $Q$ is non-constant. Hence, the leading coefficient of 	$P+Q\circ p$ coincides with the leading coefficient of  $Q\circ p$, which is irrational and we are done.

Suppose now that $\deg(Q)<\deg(P)$. Arguing by contradiction, suppose that all the non-constant coefficients of $P+Q\circ p$ are rational.
Since the coefficients of $P$ are rationally independent, the linear span over $\Q$ of the non-constant coefficients
of the  polynomial $Q\circ p= -P+(P+Q\circ p)$ has dimension equal to $\deg(P)$ over $\Q$. Furthermore,
the non-constant coefficients of $Q\circ p$ are integer combinations of the non-constant coefficients of $Q$. Hence, the dimension of the span of the non-constant coefficients of the polynomial $Q\circ p$ over $\Q$ is at most $\deg(Q)$, which is strictly smaller than $\deg(P)$ by assumption,  and we have a contradiction. This completes the proof.
\end{proof}
\begin{definition}
A {\em unipotent affine transformation} is a transformation  $T\colon \T^m\to \T^m$ of the form  $Tx=Sx+b$, $x\in \T^m$,  where $b\in \T^{m}$ and $S$ is a unipotent automorphism of $\T^{m}$.
\end{definition}
\begin{lemma}\label{L:independent}
	Let $m\in \N$ and $T\colon\T^{m}\to \T^{m}$ be an ergodic unipotent affine transformation. Then for  almost every $x\in \T^{m}$, with respect to the Lebesgue measure  in $\T^m$,  the following holds: For every non-zero $k\in \Z^m$ the non-constant coefficients of the polynomial
	$n\mapsto k\cdot T^nx$ are rationally independent.
\end{lemma}
\begin{proof}
We write $Tx=Sx+b$ where $b\in \T^{m}$ and $S$ is a unipotent automorphism of $\T^{m}$.
We are going to  prove that the conclusion holds if the coordinates of $x\in \T^{m}$ are rationally independent of the coordinates of $b$, meaning, $x$
satisfies the following property
\begin{equation}\label{E:independent}
\text{if } \, k\cdot b+\ell\cdot x=0 \bmod 1 \ \text{ for some } k,\ell\in \Z^{m}, \text{ then } \ell=0.
\end{equation}
The Lebesgue measure of the set of points $x\in\T^{m}$ with this property is $1$, and the result will follow.

	 We write $S$ also for the $(m\times m)$-matrix defining this automorphism and denote by  $S^t$ the transpose of $S$. Since $S$ is unipotent we have $(S-\id)^{m}=0$ and $(S^t-\id)^{m}=0$.
Let $K$ be the subtorus $(S-\id)\T^{m}$ of $\T^{m}$. By \cite[Theorem~4]{H},  the ergodicity of $T$ is equivalent to the ergodicity of the rotation induced by $b$ on the quotient $\T^{m}/K$, which in turn is equivalent to
\begin{equation}
\label{eq:ergo-affine}
\text{if }\ell\in\Z^{m}\text{ satisfies }(S^t-\id)\ell=0\text{ and }\ell\cdot b=0\bmod 1, \text{ then }\ell=0.
\end{equation}

Let $k\in \Z^m$ be non-zero.
 Since $(S-\id)^{m}=0$, using induction, we get that for every $n\in\N$ we have (by convention $\binom nj=0$ if $j>n$)
\begin{align}
	\notag
	k\cdot T^nx
	& = k\cdot x+ \sum_{j=1}^{m}\binom nj \bigl(k\cdot (S-\id)^jx +k\cdot(S-\id)^{j-1}b\bigr)\\
	%\notag
	%&=  k_1\cdot x+ \sum_{j=1}^{m_1}\binom nj \bigl( (S^t-\id)^jk_1\cdot x +\cdot(S^t-\id)^{j-1}k_1\cdot b\bigr)\\
	\label{eq:induction}
	&= k\cdot x+ \sum_{j=1}^{m_0}\binom nj \bigl( (S^t-\id)^jk\cdot x +(S^t-\id)^{j-1}k\cdot b\bigr)
\end{align}
where $m_0:=\min\{ q\in\N\colon (S^t-\id)^qk=0\}$.

Note that $m_0\leq m$ and that
 the   expression  in \eqref{eq:induction} is a polynomial of degree  at most $m_0$.  We will be done if we show that its non-constant  coefficients  are rationally independent. Arguing by contradiction,  suppose that
   this is not the case.
 Note  that the linear span over $\Q$ of the polynomials $n\mapsto \binom n1$, \dots, $n\mapsto \binom n{m_0}$ is equal to the space of polynomials of degree at most $ m_0$ without constant term, that is, to the linear span over $\Q$ of the polynomials $n,\dots,n^{m_0}$. Therefore, the coefficients in~\eqref{eq:induction} of the polynomials $\binom n1$, \dots, $\binom n{m_0}$ are rationally dependent and there exist integers $r_1,\dots,r_{m_0}\in \Z$, not all of them zero, such that
 %$$
 % \sum_{j=1}^{m} r_j\bigl(k_1\cdot (S-\id)^jx +k_1\cdot(S-\id)^{j-1}b\bigr)=0\bmod 1
 %$$
 %%$$
 %%k_1\cdot \sum_{j=0}^{m_1-1} r_j (S-\id)^jx +k_1\cdot \sum_{j=1}^{m_1} r_j (S-\id)^{j-1}b=0\bmod 1
 %%$$
 %that is
 $$
 \sum_{j=1}^{m_0} r_j (S^t-\id)^jk\cdot x+ \sum_{j=1}^{m_0} r_j (S^t-\id)^{j-1}k\cdot b=0\bmod 1.
 $$
 The first  sum is a linear combination with integer coefficients of the coordinates of $x$, and the second sum is a linear combination with integer coefficients of the coordinates of $b$. By our  hypothesis~\eqref{E:independent} we have
 $$
 \sum_{j=1}^{m_0}r_j(S^t-\id)^jk=0 \quad
 \text{and} \quad \sum_{j=1}^{m_0} r_j (S^t-\id)^{j-1}k\cdot b=0\bmod 1.
 $$
 We claim  that
 \begin{equation}\label{E:jm0}
 \sum_{j=1}^{m_0} r_j (S^t-\id)^{j-1}k=0.
 \end{equation}
 Indeed, writing $\ell$ for this sum we have $(S^t-\id)\ell=0$ and $\ell\cdot b=0\bmod  1$ and thus $\ell=0$ by~\eqref{eq:ergo-affine}, proving the claim.
 Let $j_0:=\min\{j\colon r_j\neq 0\}$. Applying $(S^t-\id)^{m_0-j_0}$ to both sides of \eqref{E:jm0} and recalling  that $(S^t-\id)^{m_0}k=0$, we obtain $(S^t-\id)^{m_0-1}k=0$, contradicting the definition of $m_0$. This completes the proof.	
\end{proof}
We will now combine the previous two lemmas in order to deduce an equidistribution property that will be crucial in the next subsection.
\begin{lemma}\label{L:TR}
Let $m_1,m_2\in \N$,  and  $  p\in \Z[t]$ be a polynomial  with $\deg(p)\geq 2$. Let $T\colon\T^{m_1}\to \T^{m_1}$ be an ergodic unipotent affine transformation. Then for $m_{\T^{m_1}}$-almost every $x\in \T^{m_1}$, the following holds: For every
sequence $u\colon \N\to \T^{m_2}$  with polynomial coordinates that is equidistributed in $\T^{m_2}$, the sequence $(T^nx, (u\circ p)(n))$ is equidistributed in $\T^{m_1}\times \T^{m_2}$.
\end{lemma}
\begin{proof}
Let $X_0$ be the full measure subset of $\T^{m_1}$  given by Lemma~\ref{L:independent}. We are going to  prove that the conclusion holds if  $x\in X_0$.

Note that the sequence $(T^nx, u(p(n)))$ with values in $\T^{m_1+m_2}$ has polynomial coordinates.  By the Weyl Equidistribution Theorem it suffices to show that if $k_1\in\Z^{m_1}$ and  $k_2\in\Z^{m_2}$ are not both zero, then the polynomial
\begin{equation}\label{E:up}
n\mapsto k_1\cdot T^nx+k_2\cdot u(p(n))
\end{equation}
has at least one non-constant irrational coefficient.

Suppose first that $k_1=0$, in which case we have  $k_2\neq 0$. Since the sequence $(u(n))$ is equidistributed in  $\T^{m_2}$, the polynomial $k_2\cdot u(t)$ has at least one irrational non-constant coefficient. We write this polynomial as  $t^{d+1}R_1(t)+R_2(t)$ where $R_1(t)$ has rational coefficients, $d:=\deg(R_2)$, and the leading coefficient of $R_2$ is irrational. Then
$$
k_2\cdot (u\circ p)(t)=p(t)^{d+1}(R_1\circ p)(t)+(R_2\circ p)(t).
$$
The first of these polynomials has rational coefficients and the second one has an irrational leading coefficient. Therefore,
the polynomial $k_2\cdot u\circ p$ has at least one non-constant irrational coefficient and we are done.

Suppose now that  $k_1\neq 0$. Recall that for $x\in X_0$ the non-constant coefficients of the polynomial  $n\mapsto k_1\cdot T^nx$
are rationally independent.
 Using Lemma~\ref{L:PQR}, we get that the polynomial in \eqref{E:up}
has an irrational non-constant coefficient and we are done.
This completes the proof.
\end{proof}

\subsection{A nil-equidistribution result}
We proceed now to establish the key equidistribution property on nilmanifolds needed for our purposes.

\begin{proposition}
\label{P:mainnil}
Let $(X=G/\Gamma,m_X,T)$ be an ergodic nilsystem.
Then there exists a subset $X_0$ of $X$ with $m_X(X_0)=1$ such that the following holds: For every $x\in X_0$, every nilsequence
$(\Psi(n))$, polynomial $p\in\Z[t]$ with $\deg(p)\geq 2$,  and function $f\in \CC(X)$ with $\E_\mu(f|\Krat(T))=0$, we have
	\begin{equation}\label{E:nilmain}
		\lim_{N\to\infty} \frac{1}{N}\sum_{n=1}^N f(T^nx)\cdot \Psi(p(n))=0.
	\end{equation}
\end{proposition}
\begin{remark}
It is important that the set of good $x\in X$ for which the conclusion holds is independent of the nilsequence $(\Psi(n))$.
\end{remark}
\begin{proof}
The proof is similar to the proof of~\cite[Lemma~7.6]{CFH}. It is based on an equidistribution result of Leibman~\cite[Theorem~2.17]{L} that in some cases enables us to deduce an equidistribution result for general nilsystems from the special case of unipotent affine transformations on the torus. We give a rather sketchy account below, the reader will find more details about the notions and results used in \cite[Section~7]{CFH}.

We first consider the case where $X$ is connected. In this case the rational Kronecker factor $\Krat$ of $X$ is trivial %%{\bf why}
and $\int f\,dm_X=0$.

Let $G_0$ be the connected component of $G$ and  $A_X:=X/[G_0,G_0]$ be the affine torus of $X$. We denote by  $\wt{T}$  the transformation induced by $T$ on $A_X$.
Then $\wt{T}$ is a unipotent affine transformation. %%{\bf why ?}.
 For $r\in\N$ let $B_r$ be the subset of full measure of $A_X$ associated by Lemma~\ref{L:TR} to the ergodic transformation $\wt{T}^r$ of $A_X$ and let
$$ X':=\bigcap_{r\in\N}\bigcap_{j=0}^{r-1} T^j\pi_{X}\inv(B_r)
$$
where $\pi_{X}\colon X\to A_X$ is the natural projection.
 We claim that this subset of full measure of $X$ satisfies the required property.

Note that it suffices to consider the case where $(\Psi(n))$ is a basic nilsequence. We write $\Psi(n)=\psi(b^n\cdot y)$ for some nilmanifold $Y=H/\Lambda$ with Haar measure $m_Y$, $b\in H$, $\psi\in C(Y)$, and $y\in Y$. We can assume that $\{b^n\cdot y\colon n\in\N\}$  is dense in $Y$, %%{\bf why?},
and thus that the translation by $b$ on $Y$ is ergodic.

For the moment we also assume  that $Y$ is connected. Let $A_Y:=H/[H_0,H_0]$ be the affine torus of $Y$, $\pi_{Y}\colon Y\to A_Y$ the natural projection, and $u(n):= \pi_{Y}(b^n\cdot y)$ for $n\in\N$.  Then the sequence $(u(n))$ is a polynomial sequence %%{\bf define "polynomial sequence"?}
in $A_Y$ and is equidistributed in $A_Y$.
Let $x\in X'$ and $z:=\pi_{Y}(x)$. By the definition of the set $X'$,  the sequence $(\wt{T}^nz,u(p(n)))$ is equidistributed in $A_X\times A_Y$. But $A_X\times A_Y$ is the affine torus of the connected nilmanifold $X\times Y$ and  the sequence $(\wt{T}^nz,u(p(n)))$ is the image under the quotient map $\pi_X\times\pi_Y$ of the polynomial  sequence $(T^nx,b^{p(n)}\cdot y)$ in $X\times Y$.  By~\cite[Chapter~14, Theorem~20]{HK-book} (originally established in \cite[Theorem~2.17]{L}), this sequence is equidistributed in $X\times Y$ and thus
$$
\lim_{N\to\infty}\frac{1}{N}\sum_{n=1}^Nf(T^nx)\cdot \Psi(p(n))=\lim_{N\to\infty}\frac{1}{N}\sum_{n=1}^Nf(T^nx)\cdot \psi(b^{p(n)}y)=
\int _X f\,dm_X\cdot\int_Y\psi\,dm_Y=0
$$
and we are done.

We continue to assume that $X$ is connected but we no longer  assume that $Y$ is connected. Let $Y_0$ be the connected component of $y$ in $Y$. Then there exists $r\in\N$ such that $Y_0$ is invariant and ergodic under the translation by $b^r$ and  the sets $b^j\cdot Y_0$, $j=0,\ldots, r-1$, form a partition of $Y$. %%{\bf why?}
The preceding case applied with $x$ replaced by $T^jx$, $T$ replaced by $T^r$, $Y$ replaced by $Y_0$,
 and the polynomial $p$ replaced by $p(rn+j)$, gives
$$
\lim_{N\to\infty} \frac{1}{N}\sum_{n=1}^Nf(T^{rn+j}x)\cdot \Psi(p(rn+j))=0 \quad \text{for }j=0,\dots,r-1\text{ and  }x\in X'.
$$
Averaging for $j\in\{0,\dots,r-1\}$  gives the announced result.

We consider now the case where $X$ is not connected. Then  there exist an integer $r\in\N$ and a connected subnilmanifold $X_0$ of $X$, such that the sets  $T^jX_0$, $j=0,\ldots, r-1$, form a partition of $X$, $X_0$ is invariant under $T^r$, and $(T^jX_0,T^r)$ is ergodic for $j=0,\ldots, r-1$. Writing $m_{X,j}$ for the Haar measure of $T^jX_0$,  we have $\int f\,dm_{X,j}=0$ for $j=0,\ldots, r-1$ because $\E_\mu(f|\Krat(T))=0$.
 For $j=0,\dots,r-1$, the preceding step provides a subset of full measure $X_j'$ of $T^jX_0$
and it is immediate that the union of these sets fulfills the required conditions. This completes the proof.
\end{proof}
We deduce from the previous result  and the main structural result in \cite{HK1} the following statement that is more convenient for our purposes.

\begin{corollary}
	\label{cor:nilmain}
	Let $(X,\mu,T)$ be an ergodic  system  and  $f\in L^\infty(\mu)$ with $\E_\mu(f|\Krat(T))=0$. Then there exists a subset $X_0$ of $X$ with $\mu(X_0)=1$ such that the following holds:
	For every nilsequence
	$(\Psi(n))$ and  polynomial $p\in\Z[t]$ with $\deg(p)\geq 2$  we have
	\begin{equation}\
		\label{E:nilmain2}
		\lim_{N\to\infty} \frac{1}{N}\sum_{n=1}^N f(T^nx)\cdot \Psi(p(n))=0.
	\end{equation}
\end{corollary}
\begin{proof}
	By the second part of Theorem~\ref{th:WW} we can assume that $(X,\mu,T)$ is a system of infinite order and using an approximation argument we can assume that it is a system of finite order.
	By the Structure Theorem~\cite[Chapter~18, Theorem~1]{HK-book} (originally established in \cite{HK1}), the  system $(X,\mu,T)$ is an inverse limit of a sequence $((X_j,\mu_j,T))$ of ergodic $k$-step nilsystems, with factor maps $\pi_j\colon X\to X_j$, $j \in \N$. For every $j\in\N$  there exists $f_j\in\CC(X_j)$ such that $\norm{f-f_j\circ\pi_j}_{L^1(\mu)}\to 0$ when $j\to\infty$.
	%%Note that $\Krat(X_j,\mu_j,T)$ is a finite system and that %%$\E_{\mu_j}(f_j|\Krat(X_j,\mu_j,T))$ is a continuous %%function on $X_j$.
	For $j\in \N$, substituting $\tilde{f}_j:=f_j-\E_{\mu_j}(f_j|\Krat(X_j,\mu_j,T))$ for $f_j$ and noting that  $\norm{f-\tilde{f}_j\circ\pi_j}_{L^1(\mu)}\to 0$ since $\E_{\mu_j}(f|\Krat(X_j,\mu_j,T))=0$, we can assume that $\E_{\mu_j}(f_j|\Krat(X_j,\mu_j,T))=0$.
	By the ergodic theorem, for $j\in \N$  there exists a subset $A_j$ of $X$ of  full $\mu$-measure  such that
	$$
	\lim_{N\to\infty}
	\frac{1}{N}\sum_{n=1}^N |f(T^nx)-f_j(\pi_j(T^nx))|=\norm{f-f_j\circ\pi_j}_{L^1(\mu)} \quad \text{for every }x\in A_j.
	$$
	For $j\in \N$,  let $B_j$ be the subset of full measure of $X_j$ associated by Proposition~\ref{P:mainnil} to the nilsystem $(X_j,\mu_j,T)$. Then  for every nilsequence $\Psi$ and every polynomial $p\in\Z[t]$ with $\deg(p)\geq 2$, we have
	$$
	\lim_{N\to\infty}
	\frac{1}{N}\sum_{n=1}^N f_j(T^nx)\cdot  \Psi(p(n))=0 \quad \text{for every }x\in B_j, j\in \N.$$
	Let $X_0:=\bigcap_{j\in\N} (A_j\cap \pi_j\inv(B_j))$. Then  $\mu(X_0)=1$ and every $x\in X_0$ satisfies the announced property. This completes the proof.
\end{proof}

\subsection{The rational Kronecker is characteristic}
We are now ready to verify Theorem~\ref{prop:Rationa-Kro}. We restate it  for convenience.
\begin{theorem*}
Let $T,S$ be measure preserving transformations acting on a probability space $(X,\CX,\mu)$  such that the system $(X,\mu,T)$ has zero entropy.  Let also   $p\in \Z[t]$ be a polynomial with $\deg(p)\geq 2$ and
 $f,g\in L^\infty(\mu)$ be such  that $\E_\mu(f|\Krat(T))=0$ or  $\E_\mu(g|\Krat(S))=0$. Then
$$
\lim_{N\to\infty} \frac{1}{N}\sum_{n=1}^NT^nf \cdot  S^{p(n)}g =0
$$
in $L^2(\mu)$.
\end{theorem*}
\begin{proof}
Suppose first that $\E_\mu(f|\Krat(T))=0$.
Using Proposition~\ref{prop:main} and  a standard  approximation argument we can assume that there exists $k\in\N$ such that $f$ is measurable
with respect to $\CZ_k(T)$ and $g$ is measurable with respect to $\CZ_k(S)$. Moreover, arguing as in the proof of Proposition~\ref{prop:main} we can assume that there exists a subset $X_0$ of $X$ with $\mu(X_0)=1$ such that the sequence $(g(S^nx))$ is a $k$-step nilsequence for every $x\in X_0$.

Let $\mu=\int\mu_x\,d\mu(x)$ be the ergodic decomposition of $\mu$ with respect to the transformation $T$. Then  $\mu_x(X_0)=1$ for $\mu$-almost every $x\in X$, and as we remarked in Section~\ref{SS:RK}, for $\mu$-almost every $x\in X$ we have $\E_{\mu_x}(f|\Krat(X,\mu_x,T))=0$.
%% Moreover, as we remarked in Section~\ref{SS:ok}, the system $(X,\mu_x,T)$ is an ergodic system of order $k$.
For $x\in X$,  let $X_x$ be the subset of $X$ with $\mu_x(X_x)=1$ associated by Corollary~\ref{cor:nilmain} to
the system $(X,\mu_x,T)$ and to the function $f$. Then for every  $x'\in X_0\cap X_x$ we have (here we use crucially that the set $X_x$  does not dependent on the nilsequence $(g(S^{p(n)}x'))$)
$$
\lim_{N\to\infty} \frac{1}{N}\sum_{n=1}^N f(T^nx') \cdot  g(S^{p(n)}x') =0.
$$
We deduce using  the bounded convergence theorem that for  every $x\in X_0$ we have
$$
\lim_{N\to\infty}\int\Bigl|\frac{1}{N}\sum_{n=1}^N f(T^nx') \cdot  g(S^{p(n)}x')\Bigr|^2\,d\mu_x(x')=0.
$$
Integrating with respect to the measure $\mu$ and using the bounded convergence theorem again, we obtain the required convergence.

Suppose  now that $\E_\mu(g|\Krat(S))=0$. By the first case we can assume that $f$ is measurable with respect to $\Krat(T)$. Then $f$ can be approximated in $L^2(\mu)$ by a finite linear combination of eigenfunctions associated with rational eigenvalues. Therefore, we can assume that  $f$ is an eigenfunction of this type. We are thus reduced to showing that
$$
\lim_{N\to\infty}\frac{1}{N}\sum_{n=1}^N \e(ns)\cdot g(S^{p(n)}x)=0 \  \text{ in }L^2(\mu), \quad \text{for every }s\in\Q.
$$
Since
$\E_\mu(g|\Krat(S))=0$, the spectral measure of  $g$ with respect to the system $(X,\mu,S)$ has no  rational point masses. The asserted mean convergence to zero follows  by combining this fact with the spectral theorem for unitary operators, Weyl's Equidistribution Theorem, and the bounded convergence theorem. This completes the proof.
\end{proof}

\section{Recurrence,  pointwise convergence, and positive entropy}
The goal of this section is to complete the proof of Theorem~\ref{th:recurrence}, Proposition~\ref{th:entropy}, and
Proposition~\ref{th:pointwise}.

\subsection{Proof Theorem~\ref{th:recurrence} (Recurrence)}
Let $\ve>0$. It suffices to prove that there exists $r\in\N$ such that
\begin{equation}
\label{eq:recurrence}
\lim_{N\to\infty}\frac{1}{N}\sum_{n=1}^N\mu(A\cap T^{-rn}A\cap
S^{-p(rn)}A)\geq \mu(A)^3-\ve/2.
\end{equation}
(The limit exists by  Theorem~\ref{th:convergence} applied for the transformation $T^r$ in place of $T$ and polynomial $p(rn)$ in place of $p(n)$.)
%%Indeed, suppose by contradiction that  the set
%%$\{n\in\N\colon \mu(A\cap T^{-n}A\cap S^{-p(n)}A)\geq \mu(A)^3-\ve\}$ does %%not have positive lower density. Then for every $N\in\N$ there exists an interval %%$I_N$ of $\N$, of length $N$, that does not intersect this set.
%%We let $I_N:=\{n\in\N\colon rn\in I_N\}$, $N\in\N$. Then $\Phi_N$ is an interval of %%$\N$ and the length of $\Phi_N$ tends to $\infty$ when $N\to\infty$ and thus  %%$(\Phi_N)_{N\in\N}$ is a
%%F{\o}lner sequence. We have
%%$\mu(A\cap T^{-rn}A\cap  S^{-p(rn)}A)< \mu(A)^3-\ve$ for every $n\in\Phi_N$,  %%$N\in\N$, contradicting~\eqref{eq:recurrence}.

%%So it remains to show that there exists   $r\in \N$ such %%that~\eqref{eq:recurrence} holds for every F\o lner sequence $(\Phi_N)_{N\in\N}$.
Since $\Krat(T)=\bigvee_{d\in\N}\CI(T^{d!})$ and  $\Krat(S)=\bigvee_{d\in\N}\CI(S^{d!})$,
%% there exists $r\in\N$ such that
 %%$$
%%\norm{\E_\mu({\bf 1}_A|\Krat(T))-\E_\mu(f|\CI(T^r))}_{L^1(\mu)}<\ve/4
%%$$ and this bound remains valid if $r$ is replaced with any multiple of $r$. By the same %%argument applied to $S$
we get that there exists $r\in\N$ such that
$$ \norm{\E_\mu(\one_A|\Krat(T))-\E_\mu(\one_A|\CI(T^r))}_{L^1(\mu)}<\ve/4
$$
and
$$
\norm{\E_\mu(\one_A|\Krat(S))-\E_\mu(\one_A|\CI(S^r))}_{L^1(\mu)}<\ve/4.
$$
 Note that for every $n\in \N$ the function $\E_\mu(\one_A|\CI(T^r))$ is invariant under $T^{rn}$ and the function
$\E_\mu(\one_A|\CI(S^r))$ is invariant under $S^{p(rn)}$ because $p(0)=0$ and thus $p(rn)$ is divisible by $r$.
Therefore, for every $n\in \N$ we have
\begin{multline*}
\Bigl|\int \one_A\cdot T^{rn}\E_\mu(\one_A|\Krat(T))\cdot S^{p(rn)}\E_\mu(\one_A|\Krat(S))\,d\mu
\\
-
\int\one_A\cdot \E_\mu(\one_A|\CI(T^r))\cdot \E_\mu(\one_A|\CI(S^r))\,d\mu\Bigr|\leq \ve/2.
\end{multline*}
By~\cite[Lemma~1.6]{Chu11} the last integral is greater than or equal to $\mu(A)^3$ and thus for every $n\in \N$ we have
$$
\int \one_A\cdot T^{rn}\E_\mu(\one_A|\Krat(T))\cdot S^{p(rn)}\E_\mu(\one_A|\Krat(S))\,d\mu
\geq \mu(A)^3-\ve/2.
$$
%%Let $(\Phi_N)$ be a F\o lner sequence of subsets of $\N$.
 We apply Theorem~\ref{prop:Rationa-Kro}    with $T^r$ in place of $T$ (note that $\Krat(T^r)=\Krat(T)$) and $p(rn)$ in place of $p(n)$. We get that
\begin{multline*}
\lim_{N\to\infty}\frac{1}{N}\sum_{n=1}^N\mu(A\cap T^{-rn}A\cap
S^{-p(rn)}A)=\\
\lim_{N\to\infty}\frac{1}{N}\sum_{n=1}^N
\int \one_A\cdot T^{rn}\E_\mu(\one_A|\Krat(T))\cdot S^{p(rn)}\E_\mu(\one_A|\Krat(S))\,d\mu.
\end{multline*}
Combining the above we get the  bound~\eqref{eq:recurrence}. This completes the proof. \qed

\subsection{Proof of Proposition~\ref{th:entropy} (Positive entropy)}
Let $h$ be the entropy of the system $(X,\mu,T)$. Let $\wt T$ be the shift on the sequence space $\wt X=\{0,1,2\}^\Z$ and let $\wt \mu$ be the
$(s,s,1-2s)$-Bernoulli measure on $\wt X$, where $s>0$ is small enough so that the entropy of the Bernoulli shift $(\wt X,\wt \mu,\wt T)$ is smaller than $h$. It follows from Sinai's factor theorem~\cite[Theorem 20.13]{G} %%(originally established in  \cite{S64})
 that $(\wt X,\wt\mu,\wt T)$ is a factor of $(X,\mu,T)$. Let  $\pi\colon X\to \wt X$ be the factor map.

We claim that it suffices to construct a measure preserving transformation
$\wt S$ of $\wt X$ and a subset $\wt A$ of $\wt X$ such that
$$
\wt\mu(\wt T^{-a(n)}\wt A\cap \wt S^{-b(n)}\wt A)=\begin{cases}
	0 &\text{if }n\in F, \\
	s^2&\text{ if }n\notin F.
\end{cases}
$$
Indeed, since $( X,\mu, T)$ is ergodic, it follows from Rohlin's skew-product theorem~\cite[Theorem 3.18]{G} that there exist a probability space $(Z,\nu)$ and an isomorphism of probability spaces $\Phi$ from $(\wt X\times Z,\wt\mu\times\nu)$ to $(X,\mu)$  such that $\pi(\Phi(\wt x,z))=\wt x$ for $(\wt\mu\times\nu)$-almost every $(\wt x,z)$.  We let $A:=\pi\inv(\wt A)$ and $S:=\Phi\circ(\wt S\times \id)\circ\Phi\inv$. Then  $S$ is a measure preserving transformation of $(X,\mu)$, and the required properties hold.

Therefore, we can restrict to the case where $(X,\mu,T)$ is the $(s,s,1-2s)$-Bernoulli shift on $X=\{0,1,2\}^\Z$.  We argue as in the proof of  \cite[Lemma~4.1]{FLW11}, with small changes. Given a permutation $\pi$ of $\Z$ with $\pi(0)=0$ we define
the map $\psi_\pi\colon X\to X$ by
$$
(\psi_\pi x)(n):=
\begin{cases}
	x(0) \quad  &\text{if }  n=0, \\
1-x(\pi(n)) \quad & \text{if } n\neq 0 \text{ and } x(\pi(n))\in \{0,1\}, \\
2 \quad & \text{if } n\neq 0 \text{ and } x(\pi(n))=2.
\end{cases}
$$
It is easy to verify that $\psi_\pi$ is  measure preserving, invertible, with  $(\psi_\pi)^{-1}=\psi_{\pi^{-1}}$, and $(\psi_{\pi^{-1}}x)(0)=x(0)$. We define the measure preserving transformation $S\colon X\to X$ by
$$
S:=\psi_\pi^{-1}T\psi_\pi.
$$
Then  for $n\in \N$ we have
$$
(S^nx)(0)=(T^n\psi_\pi x)(0)=(\psi_\pi x)(n)=\begin{cases}
	1-x(\pi(n))  \qquad & \text{if }  x(\pi(n))\in \{0,1\}, \\
2 \quad & \text{if }  x(\pi(n))=2.	
\end{cases}
$$
Hence, if $A:=\{x\in X\colon x(0)=1\}$ we have
$$
T^{-a(n)} A\cap S^{-b(n)}A=\{x\in X\colon x(a(n))=1,  x(\pi(b(n)))=0 \}, \qquad n\in\N.
$$
We now choose the permutation $\pi$. Since the sequences $a,b\colon \N\to \Z\setminus\{0\}$ are injective and miss infinitely many integers, we can choose $\pi$ that fixes $0$ such that  $\pi(b(n))=a(n)$ if $n\in F$ and $\pi(b(n))\neq a(n)$ if $n\notin F$. Then
$$
\mu(T^{-a(n)} A\cap S^{-b(n)}A)=
\begin{cases}
	0 \qquad & \text{if }  n\in F, \\
	s^2 \quad & \text{if }  n\notin F.	
\end{cases}
$$
This completes the proof.\qed

\subsection{Proof of Proposition~\ref{th:pointwise} (Pointwise convergence)}
We can assume that  $f,g\in L^\infty(\mu)$ are real valued. We first establish existence $\mu$-almost everywhere of the limit
\begin{equation}
\label{eq:pointwise}
\lim_{N\to\infty}\frac{1}{N}\sum_{n=1}^N f(T^nx)\cdot g(S^{p(n)}x).
\end{equation}
 By Proposition~\ref{P:ergodic}, for almost every $x\in X$ the sequence $(f(T^{n}x))$
admits correlations on the sequence of intervals $([N])_{N\in\N}$ and the corresponding Furstenberg system has zero entropy.
By our assumption, for almost every $x\in X$ the sequence $(g(S^{p(n)}x))$
admits correlations along the sequence of intervals $([N])_{N\in\N}$ and defines a
(unique) Furstenberg system.
Suppose  that $\E(g| \mathcal{Z}_\infty(S))=0$. Combining  our pointwise convergence assumption of the averages \eqref{E:assumption} with  Theorem~\ref{T:HK} and Lemma~\ref{lem:Pinsker}, we get that for almost every $x\in X$  the  $0^{\text{th}}$-coordinate projection of the Furstenberg system of the sequence $(g(S^{p(n)}x))$ is orthogonal to the Pinsker factor of this system. Hence, using Proposition~\ref{prop:Pinsker} and arguing as in Proposition~\ref{prop:char3}, we deduce that
 the averages~\eqref{eq:pointwise}
converge pointwise almost everywhere to $0$. Hence, in order to  prove pointwise convergence of the averages \eqref{eq:pointwise} we can assume that $g$ is measurable with respect to the   factor $\mathcal{Z}_\infty(S)$.

By Bourgain's  maximal inequality for
ergodic averages with polynomial iterates \cite[Theorem 6]{Bo88}, for $g\in L^2(\mu)$ we have
$$
\Bigl\Vert\sup_{N\in\N} \frac{1}{N}\sum_{n=1}^N S^{p(n)}|g|\Bigr\Vert_{L^2(\mu)}\leq C\norm{g}_{L^2(\mu)}
$$
for some universal constant $C$. Thus for $f\in L^\infty(\mu)$ we have
$$
\Bigl\Vert\sup_{N\in\N} \frac{1}{N}\sum_{n=1}^N |T^nf\cdot S^{p(n)}g|\Bigr\Vert_{L^2(\mu)}\leq C\norm f_{L^\infty(\mu)} \norm g_{L^2(\mu)}.
$$
By the argument used in the classical proof of the pointwise ergodic theorem, it follows that for $f$ fixed, the family of functions $g$ for which the convergence~\eqref{eq:pointwise} holds almost everywhere is closed in $L^2(\mu)$. Therefore, in order to prove the existence almost everywhere of the limit~\eqref{eq:pointwise} for $g$ measurable with respect to $\CZ_\infty(S)$, it suffices to restrict to the case where $g $ is measurable with respect to $\CZ_k(S)$ for some $k\in\N$.

If $g $ is measurable with respect to $\CZ_k(S)$, then the pointwise convergence of the averages \eqref{eq:pointwise} follows  by repeating the  argument used in the proof  for Proposition~\ref{prop:main}.
This finishes the convergence part of the proof.

Since the pointwise limit coincides with the $L^2(\mu)$ limit, the remaining part of the result follows from Theorem~\ref{prop:Rationa-Kro}.  This completes the proof.
\qed


\begin{thebibliography}{9999}


\bibitem{A85} {\sc  I.~Assani}. Pointwise convergence of nonconventional averages. {\em Colloq. Math..}
\textbf{102} (2005), 245--262.


\bibitem{Be85} {\sc  D.~Berend}. Joint ergodicity and mixing. {\em J. Analyse Math.}
\textbf{45} (1985), 255--284.

\bibitem{BHK} {\sc V. Bergelson, B. Host \& B. Kra.}
Multiple recurrence and nilsequences.
 {\it Invent. Math.} {\bf 160} (2005), 261--303.

\bibitem{BL96}
{\sc V.~Bergelson \& A.~Leibman}. Polynomial extensions of van der
Waerden's and Szemer\'edi's theorems.  {\em J. Amer. Math. Soc.}
\textbf{9} (1996), 725--753.


\bibitem{BL02}
{\sc V.~Bergelson \& A.~Leibman}. A nilpotent Roth theorem. {\em Invent.
	Math.} {\bf 147} (2002), 429--470.


\bibitem{BL04}
{\sc V.~Bergelson \& A.~Leibman}. Failure of Roth theorem for solvable groups of exponential growth.
{\em Ergodic Theory Dynam. Systems} \textbf{24} (2004), no. 1, 45--53.

\bibitem{Bo88} {\sc J.~Bourgain.}
On the maximal ergodic theorem for certain subsets of the
positive integers. {\em  Israel J. Math.} {\bf 61} (1988),
39--72.


\bibitem{Chu11} {\sc Q.~Chu}.
Multiple recurrence for two commuting transformations.  {\em Ergodic
	Theory Dynam. Systems} \textbf{31} (2011), no.3, 771--792.

\bibitem{CFH}
{\sc Q. Chu,  N. Frantzikinakis \& B. Host}.
Ergodic averages of commuting transformations with distinct degree polynomial iterates.
{\it  Proc. Lond. Math. Soc. (3)} {\bf 102} (2011), no. 5, 801--842.


\bibitem{CL}
{\sc J.-P. Conze, \& E. Lesigne.} Sur un théorème ergodique pour des mesures diagonales. (French) [On an ergodic theorem for diagonal measures] {\em  Bull. Soc. Math. France} {\bf 112} (1984), no. 2, 143--175.

\bibitem{dlR}
{\sc T.~de~la~Rue}. \emph{Notes on Austin’s multiple ergodic theorem}, \texttt{arXiv:0907.0538}.




\bibitem{F21}
{\sc N.~Frantzikinakis}. Furstenberg systems of Hardy field sequences and applications.   {\em J. Analyse Math.}  \textbf{147} (2022),  333--372.



\bibitem{FH18}
{\sc N.~Frantzikinakis \& B.~Host}.
The logarithmic Sarnak conjecture for ergodic weights.
{\em Ann. of Math. (2)} {\bf 187} (2018), 869--931.


\bibitem{FrK06} {\sc N.~Frantzikinakis \&  B. Kra.} Ergodic averages for
independent polynomials and applications. {\em J. London Math.
	Soc.} {\bf 74} (2006), no. 1,  131--142.

\bibitem{FLW11} {\sc  N.~Frantzikinakis, E.~Lesigne \& M.~Wierdl}.
Random sequences and pointwise convergence of multiple ergodic
averages. {\em Indiana Univ. Math. J.}  {\bf 61} (2012), 585--617.


\bibitem{Fu77}
{\sc H.~Furstenberg.}
Ergodic behavior of diagonal measures and a theorem of {S}zemer\'edi
  on arithmetic progressions.
{\em J. Analyse Math.}  {\bf 31}  (1977), 204--256.

\bibitem{Fu}
{\sc H.~Furstenberg}.
{\em Recurrence in Ergodic Theory and Combinatorial
	Number Theory.}  Princeton University Press, Princeton 1981.

\bibitem{FuK79} {\sc  H.~Furstenberg \& Y.~Katznelson}.
An ergodic Szemer\'edi theorem for commuting transformations.
{\em J. Analyse Math.} \textbf{34} (1979), 275--291.

\bibitem{FW}
{\sc H. Furstenberg \& B. Weiss.} A mean ergodic theorem for $(1/N)\sum_{n=1}^N f(T^nx)\,g(T^{n^2}x)$, {\em Convergence in ergodic theory and probability (Columbus, OH, 1993)}, Ohio State Univ. Math. Res. Inst. Publ., vol. 5, de Gruyter, Berlin, 193--227.


\bibitem{G}
{\sc E.~Glasner}. \emph{Ergodic Theory via Joinings}.
 Mathematical Surveys and Monographs, {\bf 101}. American Mathematical Society, Providence, RI, 2003.


 \bibitem{H}
{\sc  F.~Hahn}. On affine transformations of compact abelian groups.
 {\em Amer. J. of Math.}  {\bf 85} (1963), no. 3,
 428--446.

\bibitem{HK1}
{\sc B.~Host \&  B.~Kra}.
 Non conventional ergodic averages and nilmanifolds.
 {\it  Ann. of Math. (2)}  {\bf 161} (2005), 397--488.

\bibitem{HK2}
{\sc B.~Host \&   B.~Kra}.
Convergence of polynomial ergodic averages.
{\it Israel J.  Math.} {\bf 149} (2005), 1--20.

\bibitem{HK3}
 {\sc B.~Host \& B.~Kra}.
Uniformity seminorms on $\ell^\infty$ and applications.
 {\em J. Analyse Math.}  {\bf 108}  (2009), 219--276.

\bibitem{HK-book}
{\sc B.~Host \&  B.~Kra}. \emph{Nilpotent Structures in Ergodic Theory}.
Mathematical Surveys and Monographs, vol. 236. American Mathematical Society, Providence, RI, 2018.


\bibitem{KKLR}
{\sc A.~Kanigowski, J.~Ku\l aga-Przymus,  M.~Lema\'nczyk \&  T.~de~la~Rue}. On arithmetic functions orthogonal to deterministic sequences. Preprint  \texttt{arXiv:2105.11737}.

\bibitem{L}
{\sc A.~Leibman}.
Pointwise convergence of ergodic averages for polynomial sequences of translations on a nilmanifold.
{\em Ergodic Theory Dynam. Systems}  {\bf 25} (2005), no. 1, 201--213.

%%\bibitem{S64} {\sc Y. G.~Sinai}. On a weak isomorphism of transformations %%with invariant
%%measure. {\em Mat. Sb. (N.S.)} \textbf{63 (105)} (1964),  23--42.
%%English transl., {\em Amer. Math. Soc.} (2) \textbf{57} (1966), 123--143.
\bibitem{T75}
{\sc J.~P.~Thouvenot} Une classe de syst\`emes pour lesquels la
conjecture de Pinsker est vraie.  (French)
[A class of systems for which
the Pinsker conjecture is true]
{\em Israel J. Math.}  {\bf 21}
(1975), 208--214.

\bibitem{W}
{\sc M. Walsh.}
 Norm convergence of nilpotent ergodic averages. {\em Ann. of Math.} (2) {\bf 175} (2012), no. 3, 1667--1688.
\end{thebibliography}
\end{document}